\documentclass{imsart}
\usepackage{amsthm}
\usepackage{amsmath}
\usepackage{natbib}
\usepackage[colorlinks,citecolor=blue,urlcolor=blue,filecolor=blue,backref=page]{hyperref}
\usepackage{graphicx}
\usepackage{float}
\usepackage{graphicx}
\usepackage{subfig}
\usepackage{amssymb}
\usepackage{enumerate}
\usepackage{algorithm, algpseudocode, algcompatible}
\usepackage{algorithmicx}

\newtheorem{thm}{Theorem}
\newtheorem{Def}{Definition}
\newtheorem{Cor}{Corollary}
\newtheorem{Lemma}{Lemma}
\newtheorem{Assumption}{Assumption}

\renewcommand{\vec}[1]{\mathbf{#1}}
\def\X{{\vec{X}}}

\def\Y{{\vec{Y}}}

\begin{document}

\begin{frontmatter}
	\title{WIKS: A general Bayesian nonparametric index for quantifying differences between two populations}
	
	\runtitle{Bayesian Nonparametric two-sample problem}
	
	\begin{aug}
		
		\author{
			\fnms{Rafael} \snm{de Carvalho Ceregatti} \thanksref{addr1} \ead[label=e1]{rafaelceregattii@gmail.com}
		},
		\author{
			\fnms{Rafael} \snm{Izbicki}\thanksref{addr1}
			\ead[label=e2]{rafaelizbicki@gmail.com}
		}
	    \and
		\author{\fnms{Luis Ernesto} \snm{Bueno Salasar}
			\thanksref{addr1} 
			\ead[label=e3]{luis@ufscar.br}
		}
		
		\address[addr1]{
			Federal University of S\~{a}o Carlos, Rod. Washington Lu\'{i}s km 235, SP-310, S\~{a}o Carlos, SP, Brazil \\
			\printead{e1,e2,e3}
		}
		\runauthor{Ceregatti et. al.}
	\end{aug}
	
\begin{abstract} 
The problem of deciding whether two samples arise from the same distribution is often the question of interest in many research investigations. Numerous statistical methods have been devoted to this issue, but only few of them have considered a Bayesian nonparametric approach. We propose a nonparametric Bayesian index (WIKS) which has the goal of quantifying the difference between two populations $P_1$ and $P_2$ based on  samples from them. The WIKS index is defined by a weighted posterior expectation of the Kolmogorov-Smirnov distance between $P_1$ and $P_2$ and,
differently from most existing approaches, can be easily computed using any prior distribution over $(P_1,P_2)$. Moreover, WIKS is fast to compute and can be justified under a Bayesian decision-theoretic framework. We present a simulation study that indicates that the  WIKS method is more powerful than competing approaches in several settings, even in multivariate settings. We also prove that WIKS is a consistent procedure and controls the level of significance uniformly over the null hypothesis. Finally, we apply WIKS  to a data set of scale measurements of three different groups of patients submitted to a questionnaire for Alzheimer diagnostic.
\end{abstract}

	\begin{keyword}
		Bayesian nonparametrics, Hypothesis testing, Two-sample problem
	\end{keyword}
	
\end{frontmatter}

\section{Introduction}
\label{sec:intro}

The ``two-sample problem'' is a key problem in statistics and consists in testing if 
two independent samples arise from the same distribution. One way of testing such hypothesis is by making use of nonparametric two-sample tests (\citealt{Mann47,Smirnov48}). The nonparametric way of approaching the two-sample problem has been regaining a lot of interest in recent years due to its flexibility in tackling different data distributions. See for instance the methods developed in \citet{gretton2012kernel, pfister2016kernel, srivastava2016raptt, wei2016direction, ramdas2017wasserstein}.

From a Bayesian nonparametric perspective, the goodness-of-fit problem of comparing a parametric null against a nonparametric alternative has received great attention (e.g., \citealt{Florens1996,Carota96,berger2001,Basu2003}). 
However, only recently the two-sample comparison problem started been addressed.
\cite{Holmes2015} presents a closed-form expression for the Bayes factor
assuming a Polya tree prior process, rejecting the null if this statistic is below a certain threshold chosen to control the type I error. Using a similar method, but relying on a permutation approach to control the type I error, \cite{chen2014} addresses the k-sample comparison problem with censored and uncensored observations.
The latter work also uses
a Polya tree prior process.
Nonetheless, the previous methods are not easily adapted to other nonparametric priors nor to multivariate data.

In this article, we develop a novel general Bayesian nonparametric index, WIKS -- the Weighted Integrated Kolmogorov-Smirnov Statistic, that has the goal of evaluating the similarity between two groups. We show how WIKS can be used to test the equality of the two populations under a fully Bayesian decision-theoretic framework. 
The method has low computational cost and is very flexible, since it can handle any dimensionality of the observables and any nonparametric prior. In order to be implemented, it only requires the user to provide (i) a distance between probability distributions (e.g., the Kolmogorov-Smirnov distance; \citet{Kolmogorov1933}) and (ii) an algorithm to sample from the posterior distribution (e.g., the stick-breaking process in the case of a Dirichlet Process; \citealt{Sethuraman94}).

The remaining of the paper is organized as follows. In Section \ref{sec_index} we present the definition of the WIKS index, the testing procedure based on it and its decision-theoretical justification. A theoretical analysis of WIKS's properties
is shown in Section \ref{sec:asymptotic}. Section \ref{power_study} presents a simulation study designed to compare our proposal with other tests from the literature. Section \ref{sec:multivariate} shows how the index can be applied
to a multivariate setting. In Section \ref{sec:application} we apply our method to a data set on scale measurements for Alzheimer disease. Section \ref{sec:conclusions} contains our final remarks. All proofs are shown in Appendix
\ref{apendix:proofs}.

\section{The nonparametric Bayesian WIKS index}
\label{sec_index}

Assume that two independent samples $X_1, \cdots, X_n$ and $Y_1,\cdots, Y_m$ are drawn from $P_1$ and $P_2$, respectively. For a given distance $d$ between probability measures\footnote{Common choices for this metric are the Kolmogorov-Smirnov metric, the L2 metric and L\'{e}vy metric. For a survey of metrics between probability measures see \cite{Rachev2013}.}, testing the null hypothesis $H_0: P_1 = P_2$ against $H_1: P_1 \neq P_2$ is equivalent to testing $H_0: d(P_1, P_2) = 0$ against $H_1: d(P_1, P_2) > 0$. Denote by $\mathbb{P}^{x, y}$ the posterior distribution of $(P_1, P_2)$ given the observed samples $x = (x_1, \ldots, x_n)$ and $y = (y_1, \ldots, y_m)$. 

The WIKS index is
defined as follows. 
\begin{Def} \label{def_wiks}
The WIKS index against hypothesis $H_0$ is defined by
\begin{equation}
\text{WIKS}(\mathcal{D}_{n,m}) = \int_0^M w(\varepsilon) \ \mathbb{P}^{x, \, y}\big(d(P_1, P_2)>\varepsilon \big) d\varepsilon, \label{WIKS1}
\end{equation}
where $\mathcal{D}_{n,m}=\{x, y\}$ denotes the two observed samples of sizes $n$ and $m$, $w:[0, M) \longrightarrow (0, \infty]$ (the weight function) is a probability density function over $[0, M)$ and $M=\sup_{P_1,P_2}d(P_1, P_2)$ is the maximum value (possibly being $+ \infty$) of the distance $d$.
\end{Def}

\begin{figure}[H]
	\centering
	\includegraphics[width= 0.7 \linewidth]{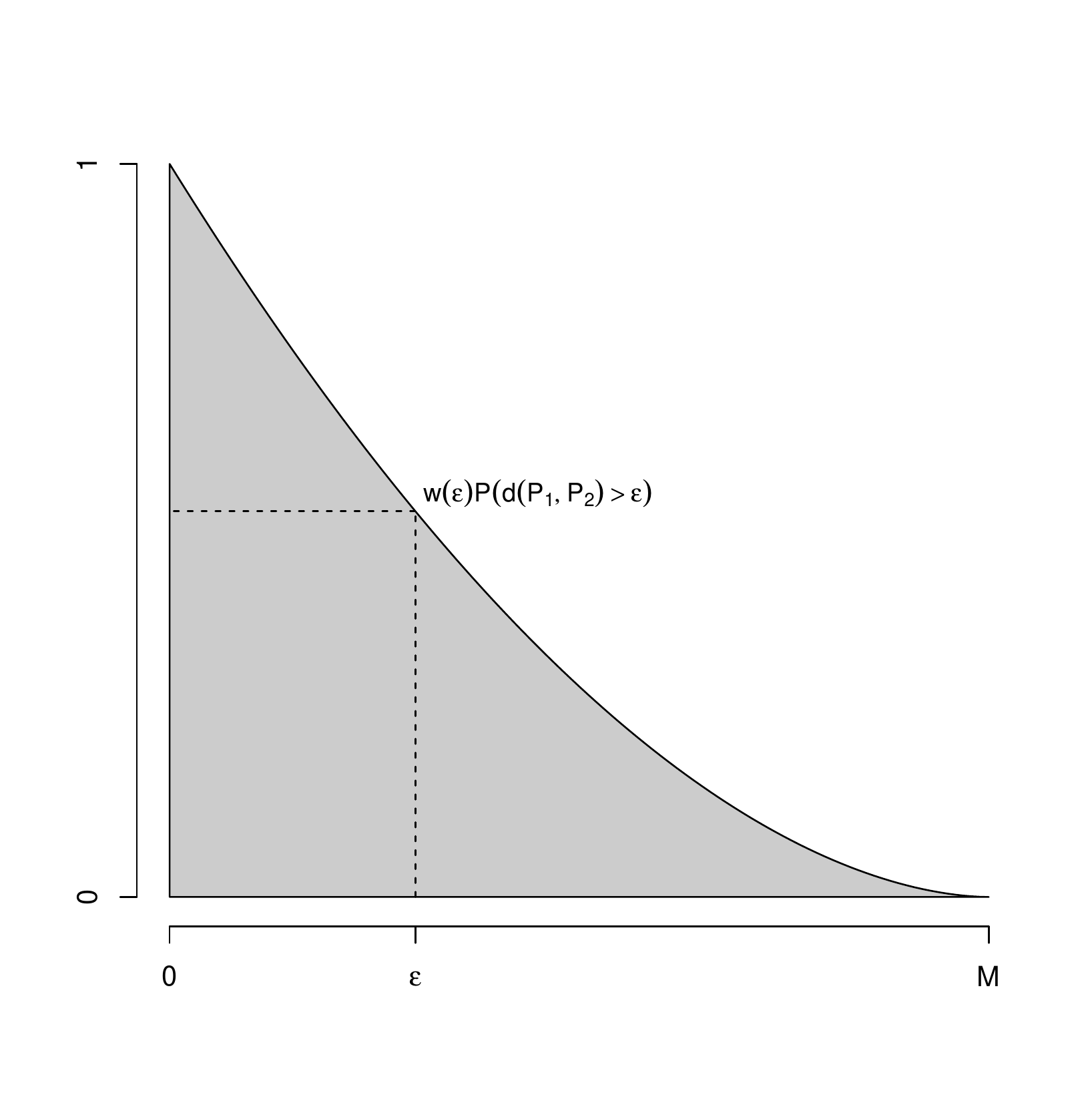}
    \caption{Geometric interpretation of the WIKS index.}
	\label{Fig-surv-dist}
\end{figure}

A geometric interpretation 
of WIKS is displayed in Figure \ref{Fig-surv-dist}.
WIKS can be thought of as a compromise between different evidence indexes against the null $H_0$. 
More specifically, a naive evidence index against the null is $P^{x,y}(d(P_1,P_2)>\varepsilon)$ for a fixed $\varepsilon > 0$, where larger values indicate greater evidence against the null. Thus, one can decide to reject the null whenever that probability exceeds a given threshold $\delta$ (e.g., 0.5)\footnote{This approach was suggested by e.g. \citet{swartz1999} in a Bayesian nonparametric goodness-of-fit context.}. However, choosing an appropriate $\epsilon$ value 
is typically not easy, especially in a nonparametric framework. Moreover, it can
also lead to inconsistent decisions: for instance, suppose that the actual distance between $P_1$ and $P_2$ is  $\varepsilon'$ in $(0, \varepsilon)$, then $P^{x,y}(d(P_1,P_2)>\varepsilon)$ converges to $0$ as the sample sizes increase (since the posterior of $d(P_1,P_2)$ converges to $\varepsilon'$) leading one to wrongly accept the null. Instead of fixing an $\varepsilon$ value, WIKS combines all the evidences $P^{x,y}(d(P_1,P_2)>\varepsilon)$ for different $\varepsilon$ using the weighted average given in \eqref{WIKS1}. Notice that, by choosing a constant weight function $w$, WIKS index \eqref{WIKS1} is proportional to the area below the survival curve of $d(P_1, P_2)$, which is the posterior expected value of $d(P_1, P_2)$. Different choices of the weight function can be considered depending on the specifics of the problem at hand.

Next, we investigate some properties of WIKS.
 
\begin{thm} \label{teo_wiks}
Let $\mathbb{E}^{x, y}$ 
denote the expectation with respect to $\mathbb{P}^{x, y}$. Then,
\begin{equation}
\text{WIKS}(\mathcal{D}_{n,m}) = \mathbb{E}^{x, y}\big[W(d(P_1, P_2)\big], \label{WIKS2}
\end{equation}
where $W$ is the cumulative distribution of the weight function $w$. 
\end{thm}

Theorem \ref{teo_wiks} shows that WIKS  can be expressed as the expected value of $W(d(P_1, P_2))$
with respect to the posterior distribution.
This implies that  a Monte Carlo approximation for WIKS is readily available from posterior simulations of $(P_1, P_2)$. A description of such procedure is given in Algorithm \ref{alg1}.

\begin{algorithm}[H]
\caption{WIKS computation} \label{alg1}
\begin{algorithmic}[1]
  \REQUIRE {samples $x$ and $y$ of sizes $n$ and $m$; posterior distribution   $\mathbb{P}^{x, \, y}(P_1,P_2)$; \linebreak cumulative weight function $W$; number of Monte Carlo simulations $S$}
  \ENSURE $\text{WIKS}(\mathcal{D}_{n,m})$
  \vspace*{0.5 cm}
  \STATE Sample $(P_{1, 1}, P_{2, 1}), \ldots, (P_{1, S}, P_{2, S})$ independently from the posterior distribution $\mathbb{P}^{x, y}$;
  \STATE Approximate WIKS index by
  $$
  \text{WIKS}(\mathcal{D}_{n,m}) \approx \dfrac{1}{S} \sum_{s=1}^S W(d(P_{1, s}, P_{2, s}))
  $$
\end{algorithmic} 
\end{algorithm}

WIKS also has desirable properties for an index against the null hypothesis, which are presented in Theorem \ref{index.proper}.

\begin{thm} \label{index.proper}
WIKS satisfies:
\begin{enumerate}[(a)]
\item $0 \leq \text{WIKS}(\mathcal{D}_{n,m}) \leq 1$ for any observed sample $\mathcal{D}_{n,m}$;
\item $\text{WIKS}(\mathcal{D}_{n,m}) = 0$ if, and only if, $d(P_1, P_2) = 0$ almost surely;
\item $\text{WIKS}(\mathcal{D}_{n,m}) = 1$ if, and only if, $d(P_1, P_2) = M$ almost surely;
\item $\text{WIKS}(\mathcal{D}_{n,m})$ is increasing with respect to $d(P_1, P_2)$.
\end{enumerate}

\end{thm}

\subsection*{Decision-theoretic justification} \label{subsec::decision}

From the above, a natural decision criterion should be to reject $H_0$ whenever $\text{WIKS}(\mathcal{D}_{n,m}) > c$, for a given threshold $c$. Indeed, this procedure can be justified under a Bayesian decision framework \citep{Degroot}. In fact, let us consider $\mathbb{A} = \{0, 1\}$ the decision space, where $0$ stands for $\text{accepting } H_0$ and $1$ for $\text{rejecting } H_0$, and the loss function
\begin{align}
\label{eq:loss}
L\big((P_1, P_2), a\big) = \left\{ \begin{array}{ll}
c_0 W(d(P_1, P_2)), & \mbox{ if $a = 0$},\\
c_1[1-W(d(P_1, P_2))], & \mbox{ if $a = 1$},\end{array} \right.
\end{align} 
where $c_0$ and $c_1$ are positive real numbers representing the maximum loss when accepting and rejecting $H_0$, respectively. Observe that, if we decide to accept $H_0$, the loss function is zero if $d(P_1, P_2) = 0$ and increases with the value of $d(P_1, P_2)$. On the other hand, if we decide to reject $H_0$, then the function decreases with the value of the distance $d(P_1, P_2)$ and vanishes if $d(P_1, P_2)$ is the maximum possible value $M$. Next theorem shows that the Bayes decision is to reject the null hypothesis when WIKS is large enough.

\begin{thm} \label{dt}
The Bayes rule for the loss function \eqref{eq:loss} is given by rejecting $H_0$ if
\begin{equation}
\label{cutoff}
\text{WIKS}(\mathcal{D}_{n,m})  > c,
\end{equation}
where $c = c_1/(c_1 + c_0)$.
\end{thm}

\section{Asymptotic properties}
\label{sec:asymptotic}
In this section we prove
that (i) the distribution of
WIKS is approximately invariant over
$H_0$, (ii) the WIKS statistic is consistent, and (iii)
the hypothesis testing procedure
based on WIKS is consistent.
We make the following assumptions:

\begin{Assumption}
	\label{assump::bounded}
$P_1, P_2 \sim DP(K, G)$
are
 independent Dirichlet processes, and  
there exists a measure  $\nu_1$
	that dominates
	$G$ such that $g(x):=\frac{dG}{d\nu_1}(x)\leq C$ for some $C>0$.
\end{Assumption}

\begin{Assumption}
	\label{assump::uniform}
$W(\epsilon)=\epsilon,$ $0 \leq \epsilon \leq 1$,
i.e., a uniform weighting is used for WIKS.
\end{Assumption}

Let
$\mathbb{D}_{n,m}=\{X_1, \ldots, X_n, Y_1, \ldots, Y_m\}$
and 
$$
Z^{n,m}(\mathbb{D}_{n,m}) = \sup_{x \in \mathbb{R}} \Bigg \lvert\dfrac{1}{K + n} \sum_{i=1}^n I_{(-\infty, x]}(X_i) - \dfrac{1}{K + m} \sum_{j=1}^m I_{(-\infty, x]}(Y_j) \Bigg \rvert.
$$

The following corollaries are
proven in Appendix \ref{apendix:proofs}.

\begin{Cor}[Approximate invariance over $H_0$]
\label{cor:invariant}
Under $H_0: P_1=P_2$ and under Assumptions \ref{assump::bounded}
and \ref{assump::uniform}, 
if $n,m\longrightarrow \infty$ in a way such that 
$m/(n+m)\longrightarrow \tau$ for some $0<\tau<1$, then
for every $0 \leq x \leq 1$
$$P\left(\mbox{WIKS}(\mathbb{D}_{n,m}) \leq F^{-1}_{Z^{n,m}(\mathbb{U})}(x) \right)  \xrightarrow{n,m \longrightarrow \infty} x,$$
where 
$Z^{n,m}(\mathbb{U})$ is the distribution of 
$Z^{n,m}(\mathbb{D}_{n,m})$ when both samples come from a
$\mbox{Unif}(0,1)$ distribution.
\end{Cor}

In words, Corollary \ref{cor:invariant}
shows that, if $P_1=P_2$, the distribution of  
$\mbox{WIKS}(\mathcal{D}_{n,m})$ does not depend asymptotically on the value of $P_1$. This implies that 
the procedure
the hypothesis test
described in Theorem \ref{dt} approximately controls the level of significance uniformly over $H_0$.

\begin{Cor}[Consistency of the WIKS statistic]
\label{cor::consist}
Denote by $H_X$ and $H_Y$
the cumulative distribution functions of $P_1$
and $P_2$, respectively.
Under Assumptions \ref{assump::bounded}
and \ref{assump::uniform},
$$\mbox{WIKS}(\mathbb{D}_{n,m})  \xrightarrow[n,m \longrightarrow \infty]{a.s.}  \sup_{x \in \mathbb{R}}|H_X(x)-H_Y(x)|$$
\end{Cor}

Corollary \ref{cor::consist} implies that under $H_0$, WIKS index converges to zero
as the sample size grows, while if $P_1 \neq P_2$, it converges to a strictly positive number, which is the Kolmogorov-Smirnov distance between the two cumulative distribution functions.

\begin{Cor}[Consistency of the test procedure]
\label{cor::consist2}
Consider the hypothesis test procedure given by
$$\phi_{n,m}(\mathbb{D}_{n,m})=1 \iff \mbox{WIKS}(\mathcal{D}_{n,m}) \geq F^{-1}_{Z^{n,m}(\mathbb{U})}(1-\alpha),$$
where 
$Z^{n,m}(\mathbb{U})$ is the distribution of 
$Z^{n,m}(\mathbb{D}_{n,m})$ when both samples come from a
$\mbox{Unif}(0,1)$ distribution.
Assume that
$n,m\longrightarrow \infty$ in a way such that 
$m/(n+m)\longrightarrow \tau$ for some $0<\tau<1$.
Under Assumptions \ref{assump::bounded}
and \ref{assump::uniform},
$$P\left(\phi_{n,m}(\mathbb{D}_{n,m})=1\right) \xrightarrow{n,m \longrightarrow \infty} \alpha$$
if $H_0$ holds and 
$$P\left(\phi_{n,m}(\mathbb{D}_{n,m})=1\right) \xrightarrow{n,m \longrightarrow \infty} 1$$
if $H_1$ holds.
\end{Cor}

Corollary \ref{cor::consist2}
show how the threshold of the hypothesis test
described in Theorem \ref{dt} can be chosen if
one desires to control  its level of significance. Moreover, it shows that this test procedure
is consistent, in the sense that
with high probability (for large sample sizes) it leads to the rejection of $H_0$ if $H_1$ holds.

\section{Power Function Study}
\label{power_study}

In this section, we perform a simulation study to compare the frequentist performance of the WIKS procedure with the well-established Kolmogorov-Smirnov (KS) and Wilcoxon (WILCOX) tests, and also with the testing procedure proposed by \cite{Holmes2015} (HOLMES), which considers the Polya tree process prior. For all tests, a nominal level $\alpha = 0.05$ was considered. The numerical calculations for the KS and WILCOX methods were performed using the standard outputs of the R \citep{R} commands ``ks.test'' and ``wilcox.test'' provided in the `stats' package. For the HOLMES method, we used the code provided by the authors at \url{http://www.stats.ox.ac.uk/~caron/code/polyatreetest/demoPolyatreetest.html}.

\subsection*{WIKS and HOLMES decision procedures}

The WIKS index is determined by the specification of a prior distribution for $P_1$ and $P_2$, a metric $d$ and a weight function $w$. In the following, we consider that $P_1$ and $P_2$ follow two independent and identical Dirichlet process $DP(K, G)$, with $K > 0$ being the concentration parameter and $G$ the base probability distribution with support in a subset $S$ of the real line $\mathbb{R}$. The probability $G$ is seen as an initial guess for the distribution of the data, while  the concentration parameter $K$ is a degree of confidence in that distribution. We set $K = 1$ and $G$ equal to the $N(0,1)$ distribution. The chosen metric $d$ is the Kolmogorov (also know as uniform) distance defined by $d(P_1, P_2) = \sup_x |P_1((-\infty, x]) - P_2((-\infty, x])|$ and, since the maximum of this distance is $1$, the weight function $w$ is taken to be the density function of a $\text{Beta}(1, 4)$ density, which has cumulative weight function $W_{\lambda}(t)=1 - (1-t)^{4}$, $t \in [0,1]$\footnote{In fact, any choice of $W_{\lambda}(t)=1 - (1-t)^{\lambda}$ with $\lambda > 0$ can be made and give similar results.}.

To define a decision rule using WIKS, we need to choose the threshold value $c$ given in \eqref{cutoff}. This choice can be made by interpreting the roles of the constants $c_0$ and $c_1$ in the loss function \eqref{eq:loss}, but this assessment is not straightforward. In this work, we adopt a different approach called by \cite{Good92} a ``bayes / non-bayes compromise'', which consists in choosing the threshold value $c$ that controls the type I error at level $\alpha$. 

In \cite{Holmes2015}, the index used to reject the null is the logarithm of the Bayes factor (LBF) for model comparison, where smaller values indicates greater evidence against $H_0$, that is, the decision procedure is to reject the null whenever $LBF < h$ for some threshold $h$. Under each hypothesis $H_0$ and $H_1$, the authors assume Polya tree process priors with a standard gaussian N(0,1) as a centering distribution for the partition specification (see \citealt[Section 3.1]{Holmes2015} for more details). Before applying the decision procedure, the data is standardized by the median and the interquartile range of the aggregated data $x$ and $y$. The threshold $h$ is also chosen using the  ``bayes / non-bayes compromise''.

In practice, we  obtain the threshold value of WIKS and HOLMES by simulating $R$ replicates of samples $x$ and $y$ (with sizes $n$ and $m$) from the same distribution $P$, calculating the index value for each replicate and then choosing the threshold value as the $1 - \alpha$ (or $\alpha$ for HOLMES) sample quantile of the $R$ index values. Notice that  Corollary \ref{cor:invariant} implies that the quantile estimate for WIKS is approximately invariant with respect to the choice of $P$.
Thus, although the threshold is obtained for a particular value
of $P$, WIKS controls type I error uniformly over $H_0$.

For both WIKS and HOLMES, we choose $P$ as the N(0,1) distribution and $R = 1,000$. The threshold values obtained were $0.7270$ for WIKS and $-0.8572$ for HOLMES.

\subsection*{Simulation Study}

We estimate the power of each method  under $8$ scenarios by simulating $1,000$ data sets $\textbf{X} = (X_1, \ldots, X_{50})$ and $\textbf{Y} = (Y_1, \ldots, Y_{50})$. The scenarios were chosen to express different types of deviations from the null, with larger values of $\theta$ representing greater deviations:
\begin{enumerate}
\item Normal Mean Shift: $\textbf{X} \sim N(0,1)$ and $\textbf{Y} \sim N(\theta,1)$, $\theta =0,\cdots, 3$
\item Normal Variance Shift: $\textbf{X} \sim N(0,1)$ and $\textbf{Y} \sim N(0, \theta)$, $\theta =1,\cdots, 4$
\item Lognormal Mean Shift: $\log\textbf{X} \sim N(0,1)$ and $\log\textbf{Y} \sim N(\theta, 1)$, $\theta = 0, \cdots, 3$
\item Lognormal Variance Shift: $\log\textbf{X} \sim N(0,1)$ and $\log\textbf{Y} \sim N(0, \theta)$, $\theta = 1, \cdots, 5$
\item Beta Symmetry: $\textbf{X} \sim Beta(1,1)$ and $\textbf{Y} \sim Beta(\theta,\theta)$, $\theta =1,\cdots, 6$
\item Gamma Shape: $\textbf{X} \sim Gamma(3,2)$ and $\textbf{Y} \sim Gamma(\theta,2)$, $\theta =3,\cdots, 6$
\item Normal Mixtures: $\textbf{X} \sim N(0,1)$ and $\textbf{Y} \sim \frac{1}{2}N(-\theta,1)+\frac{1}{2}N(\theta,1)$, $\theta = 0, \cdots, 3$
\item Tails: $\textbf{X} \sim N(0,1)$ and $\textbf{Y} \sim t(\theta^{-1})$, $\theta = 10^{-3}, \cdots, 10$.
\end{enumerate}

\begin{figure}[hp!]
	\centering
	\includegraphics[width= 0.75\linewidth]{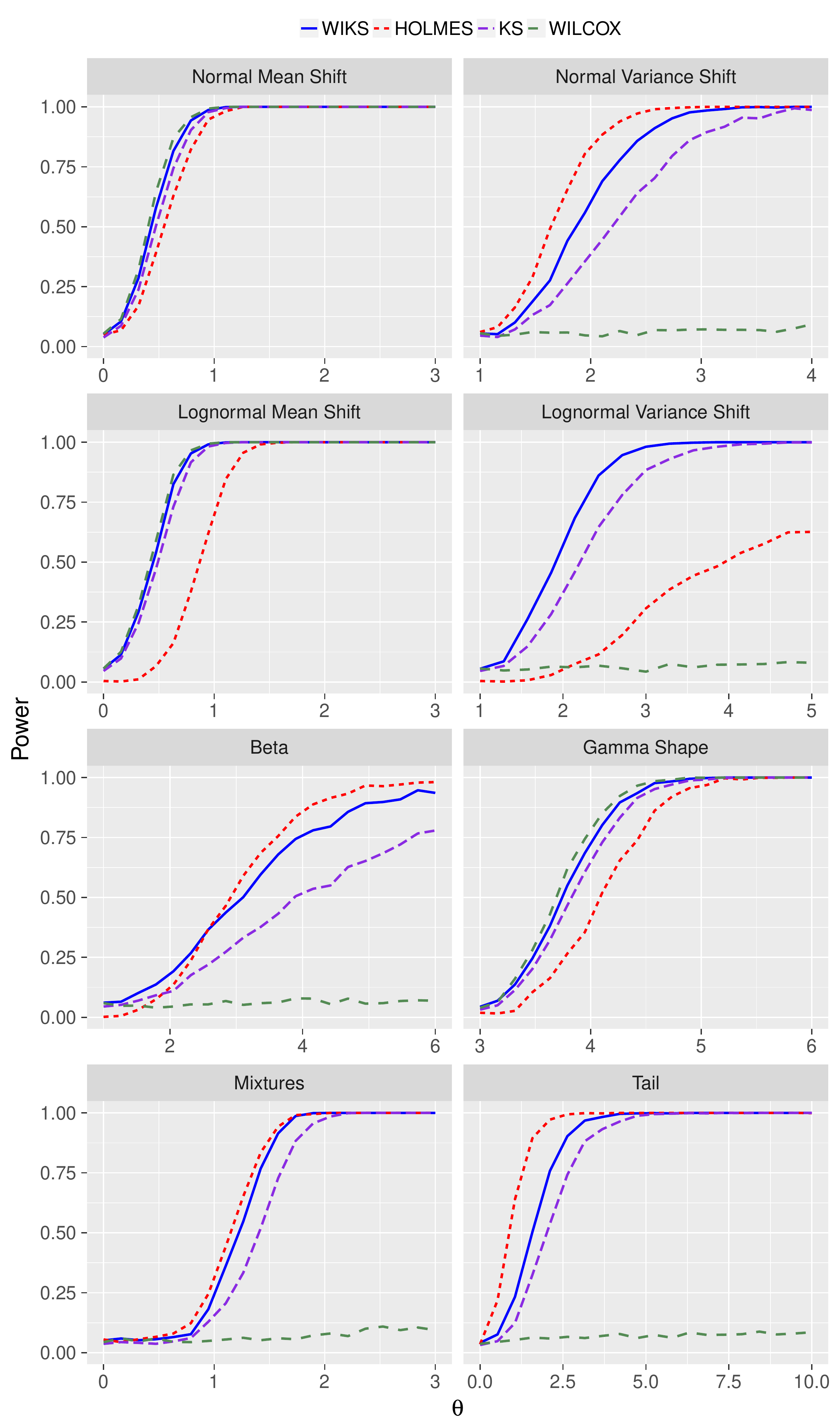}
	\caption{Power function comparison of WIKS, HOLMES, KS and WILCOX methods under  $8$ scenarios.}
	\label{Fig_power}
\end{figure}

Figure \ref{Fig_power} indicates that WIKS is very competitive in all scenarios, having uniformly greater power than KS in all situations. Also, WIKS outperforms WILCOX in all scenarios except scenarios 1, 3 and 6 (Normal Mean Shift, Lognormal Mean Shift, Gamma Shape), where they present very similar performance. When compared to HOLMES, WIKS has greater performance for scenarios 3, 4, 6 (Lognormal Mean Shift, Lognormal Variance Shift, Gamma Shape) and HOLMES wins in scenarios 2 and 8 (Normal Variance Shift, Tails). In the remaining settings both methods are comparable. 

It is also interesting to note the role of the invariance property of WIKS (Corollary \ref{cor:invariant}): while WIKS has power at the null very close to the nominal $\alpha = 0.05$
in all settings, the power of HOLMES at the null is much lower than the nominal for the settings $3$ to $6$. Possibly, this is because the support of the distribution of the data is different from the centering distribution N(0,1) used in the Polya tree process prior. Further, the latter issue implies that the cutoff determination of HOLMES is be very sensitive to the choice of the null distribution $P$ used to obtain it. To illustrate this, we present at Table \ref{table:cutoffs} the cutoff values of both methods obtained for $\alpha = 0.05$ and data $x$ and $y$ generated from the N(0,1), the U(0,1) and the LN(0,1) distributions (under the null). While WIKS cutoff values are roughly constant, HOLMES cutoff values are quite unstable. Thus, different choices of $P$ to determine the cutoff for HOLMES can cause the true level to be much larger or much smaller than the nominal.
\begin{table}[H]
\centering
\caption{Cutoffs of each method to control the type I error probability at 5\% when $x$ and $y$ are generated from the same distribution.}
\label{table:cutoffs}
\begin{tabular}{llll}
\hline
        & $N(0,1)$ &  $U(0,1)$  &  $LN(0,1)$ \\ \hline 
WIKS    & 0.7270   &  0.7337    &   0.7302  \\
HOLMES  & -0.8572  &  2.0144    &   2.7511 \\
\hline
\end{tabular}
\end{table} 

\section{Multivariate two-sample testing}
\label{sec:multivariate}

WIKS can be extended to other settings. We now explore how to use it  to compare two populations with respect to multivariate distributions.
We also explore  the fact that
the index can be computed
using \emph{any} prior
probability over the parameter space, and not only the Dirichlet process.

Let
$\X_1,\ldots,\X_n$
be a multivariate i.i.d. random 
vector drawn from $P_1$ and
$\Y_1,\ldots,\Y_m$ be a
multivariate i.i.d. random 
vector drawn from $P_2$.
Our goal is to test
$H_0: P_1=P_2$.
We assume that $\X_i,\Y_i \in \mathbb{R}^d$.
Let $d(P_1,P_2)$  be a distance between the multivariate
distributions $P_1$ and
$P_2$. For instance,
$d(P_1, P_2)$ may be the
multivariate Kolmogorov metric, defined by 
$$d(P_1, P_2) = \sup_{x_1,\ldots,x_d}|P_1\left(\Pi_{i=1}^d(-\infty, x_i]\right) - P_2\left(\Pi_{i=1}^d(-\infty, x_i]\right)|.$$

We use the same formulation of
WIKS
as described in Section  \ref{sec_index} to test $H_0$, i.e., $WIKS(\mathcal{D}_{n,m}) = \mathbb{E}^{x, y}\big[W(d(P_1, P_2)\big]$.
Notice that 
the distance function 
$d(P_1,P_2)$ is still a (real) random variable, and therefore the weighting function
$w(\epsilon)$  has the same interpretation
as before. 

 Figure \ref{fig::extension}
compares the power  of WIKS
against the KDE test for multivariate
two-sample testing \citep{duong2012closed}.
In this experiment,
the first sample consists in 100
sample points from a
$N((0,0),\Sigma)$ distribution. The second sample consists in 
100 sample points from a $N((\theta,\theta),\Sigma)$. 
While the left panel uses the covariance matrix
\[
\Sigma=\begin{bmatrix}
   1       & 0 \\
    0     & 1
\end{bmatrix}
\] 
the right panel
consists in using
\[
\Sigma=\begin{bmatrix}
   1       & 0.5 \\
    0.5     & 2
\end{bmatrix}
\]

Two versions of WIKS are used:
the first one uses a Dirichlet Process
as a prior for $P_1$ and $P_2$
with two independent standard gaussian distributions as a base measure
and $K=1$. The second versions uses a mixture of Gaussians as a prior for $P_1$ and $P_2$, with the default values from package \verb|mixAK| \citep{komarek2014mixak}.
All thresholds of the decision procedures were chosen
so as to guarantee a significance level of 5\%.

The figure shows that both versions of WIKS have better performance than the
KDE test in both settings, which suggests that WIKS is a promising approach for multivariate 
two-sample testing.
Moreover, in this case both
prior distributions lead to similar results, with the
mixture of Gaussians being
marginally better.

\begin{figure}[hp!] \centering
	 \subfloat[Bivariate example  without covariance]{\includegraphics[scale=0.33]{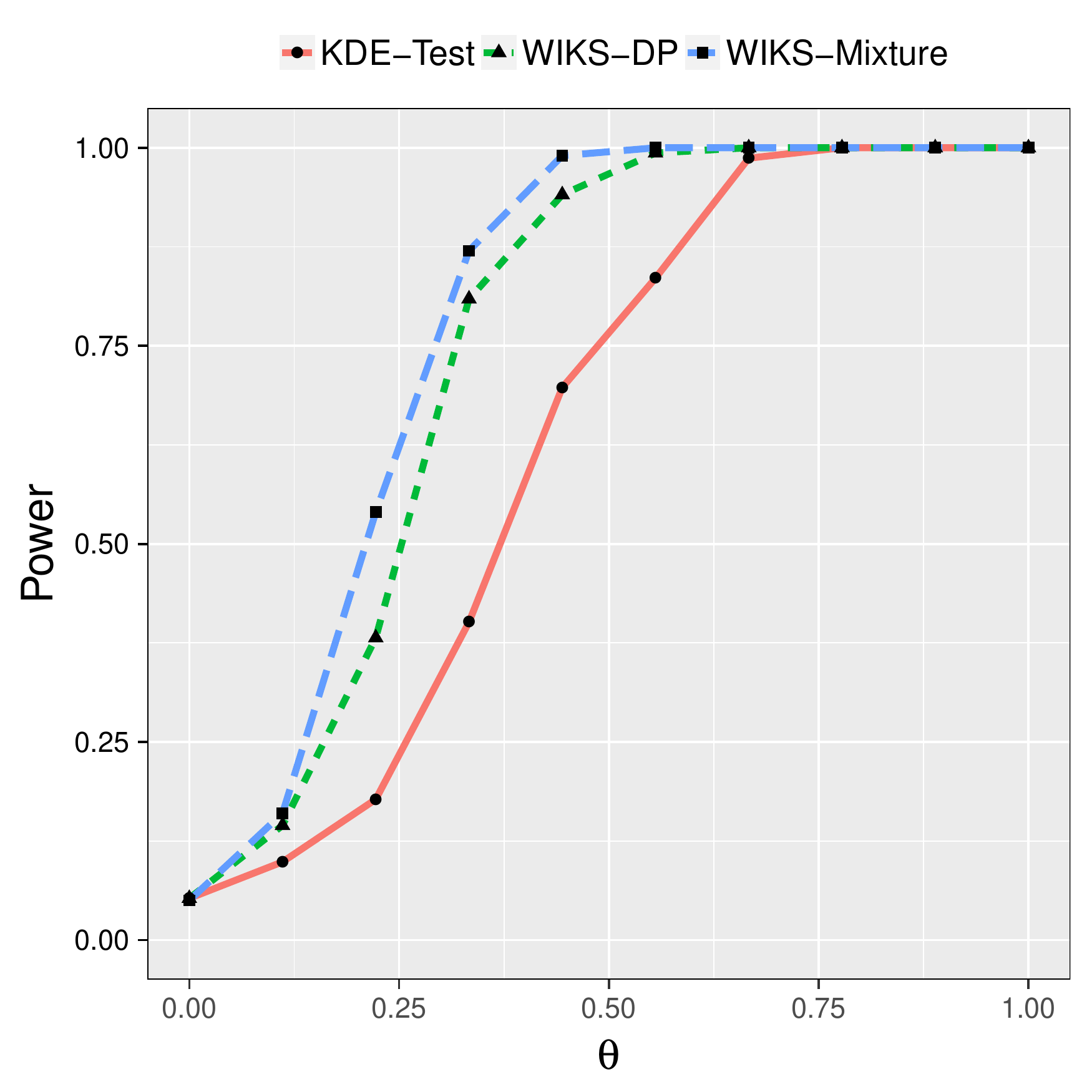}} 
	 \subfloat[Bivariate example with covariance]{\includegraphics[scale=0.33]{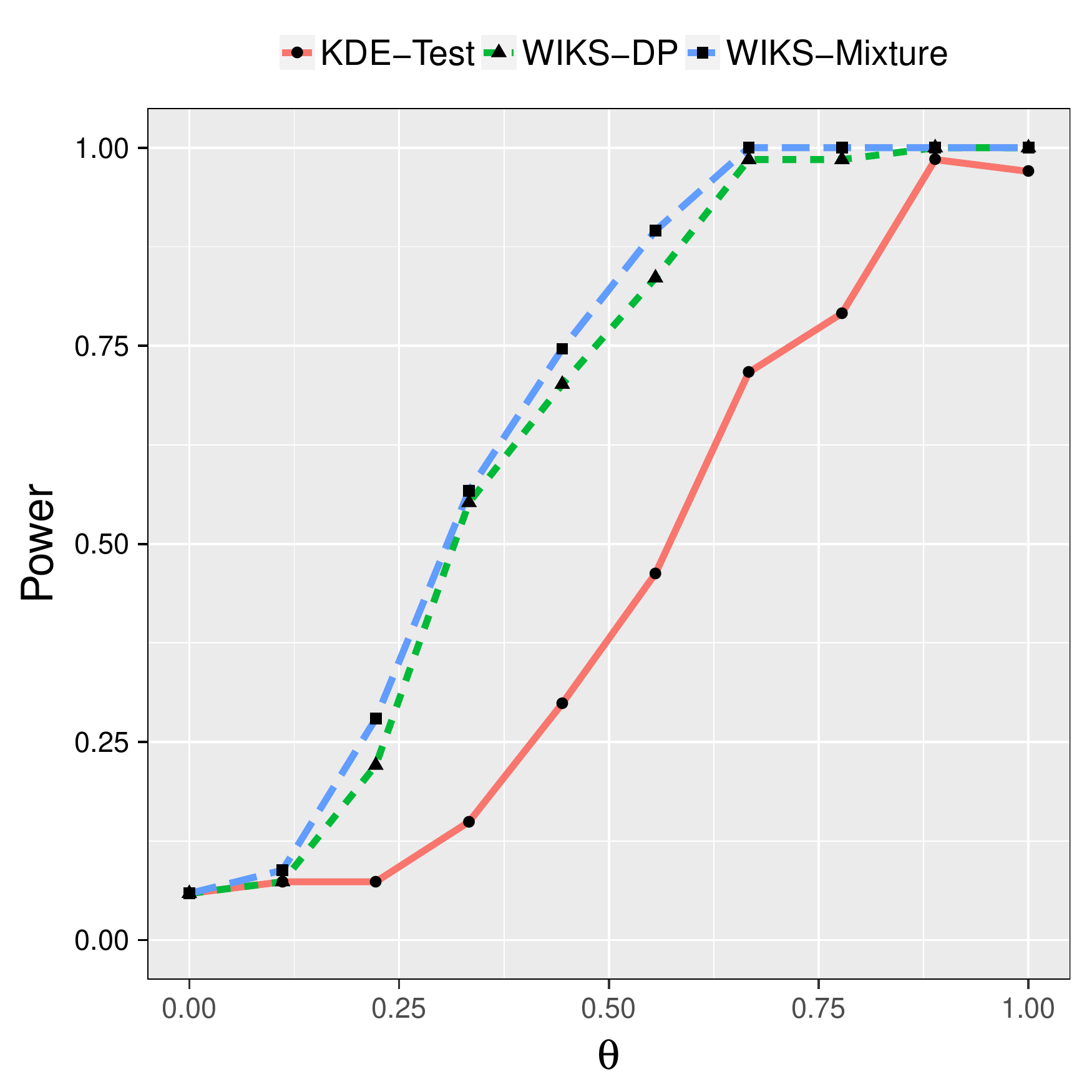}}     
	 \caption{Power comparison of two-sample multivariate testing.}
	 \label{fig::extension}
\end{figure}

\section{Application} \label{sec:application}

We apply our methods to a data set of three groups of patients (CG: the control group, MCD: with mild cognitive decline and AD: with Alzheimer's disease) submitted to a questionnaire for Alzheimer's disease diagnostic (CAMCOG). More details on this dataset can be obtained in \cite{Cecato2016}. The main idea is to quantify the differences between the groups using our methods.

\begin{figure}[hp!]
	\begin{center}
		\includegraphics[width = .8\linewidth]{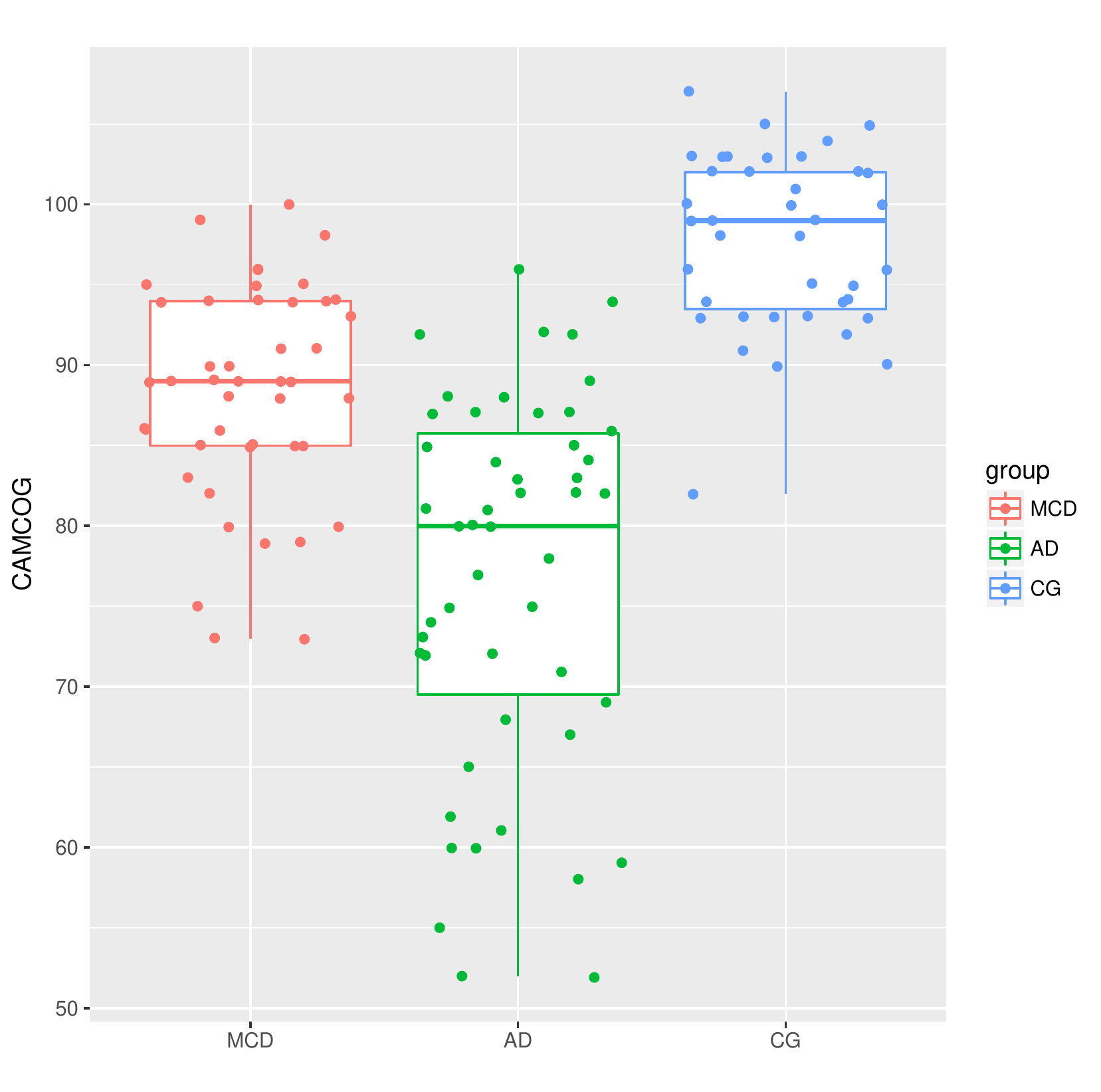}
		\caption{Boxplot of CAMCOG scores for the groups MCD, AD and CG.}
		\label{anadesdados}
	\end{center}
\end{figure}

 Figure \ref{anadesdados} shows  that all groups present different behavior with respect the score obtained from CAMCOG. The group with Alzheimer's disease (AD) has the lowest CAMCOG scores and the control group (CG) the highest ones. The group with mild cognitive decline (MCD) has score values in-between the other two groups. Thus, it is expected that the WIKS index will be greater when comparing AD and CG groups than for the other comparisons. In fact, for AD vs CG, CG vs MCD and MCD vs AD the WIKS index are $0.9993, 0.9629, 0.9312$ with respective thresholds $0.7558$, $0.7681$ and $0.7314$, leading to the rejection of  null for all pairwise comparisons. From this analysis, we conclude that CAMCOG is an useful tool for initial diagnostic of Alzheimer disease, being able to properly distinguish between the three groups.

\section{Conclusions} \label{sec:conclusions}

We propose a method to compare two populations $P_1$ and $P_2$ that relies on a Bayesian nonparametric discrepancy index (WIKS) defined as a weighted average of the posterior survival function of the Kolmogorov distance $d(P_1, P_2)$. The WIKS index can also be expressed as the posterior expectation in terms of $d(P_1, P_2)$, which makes it easier to compute its value  using  samples of the posterior distribution.  The WIKS definition can be seen as an aggregated evidence against the null and the proposed decision procedure is the Bayes rule under a suitable loss function. A key advantage of WIKS method is that it controls the type I error probability uniformly over $H_0$. Moreover, we proved that the proposed WIKS statistic and the decision procedure are both consistent.

In a power function simulation study, WIKS presents better performance than the well-established Wilcoxon and Kolmogorov-Smirnov tests. When compared to the method proposed by \citet{Holmes2015}, WIKS shows similar performance in many settings and is superior when the support of data are restricted to the positive real numbers or the unitary interval. For a data-set on questionaire scores used for Alzheimer diagnose applied to 3 groups,
WIKS could correctly indentify the difference between the groups.

We conclude that  WIKS is a powerful and  flexible method to compare populations with
low computational cost. Even thought we have chosen the Dirichlet Process as our prior, any other nonparametric (e.g, the Polya tree or the Beta processes) or even parametric prior could be used without the need of adjustments: WIKS computation only requires a sampling algorithm for posterior simulation. Moreover, the dimensionality of data poses no restriction to the method, since it is based on the concept of distances, which always take values on the real line. Further investigation is needed to assess the effect of the choices of the metric $d$ and the weight function $w$ on the performance of the method. Future research directions are extending the methods presented here to goodness-of-fit problems and investigating the performance in high-dimensional settings.

\bibliographystyle{ba}
\bibliography{references.bib}
	
\section*{Acknowledgements}
	This work was partially supported by
	FAPESP grant  2017/03363-8.
\section*{Acknowledgements}

\appendix	
    
\section{Proofs}  
\label{apendix:proofs}  

\begin{proof}[Proof of Theorem \ref{teo_wiks}]
Let $\mathbb{P}_d$ be the probability distribution of $d(P_1, P_2)$ assuming that $(P_1, P_2)$ is distributed according to $\mathbb{P}^{x, y}$. Thus,
\begin{equation*}
WIKS(\mathcal{D}_{n,m}) = \int_0^M w(\varepsilon) \mathbb{P}_d\big((\varepsilon, M]\big) = \int_0^M \int_0^M w(\varepsilon) I_{(\varepsilon, M]}(z) d\mathbb{P}_d(z) d\varepsilon,
\end{equation*}
which implies by the Fubini theorem, that
\begin{eqnarray*}
WIKS(\mathcal{D}_{n,m}) &=& \int_0^M \int_0^M w(\varepsilon) I_{(\varepsilon, M]}(z) d\varepsilon d\mathbb{P}_d(z), \\
&=& \int_0^M \int_0^z w(\varepsilon) d\varepsilon d\mathbb{P}_d(z) \\
&=& \int_0^M W(z) d\mathbb{P}_d(z) \\
&=& E[W(d(P_1, P_2))],
\end{eqnarray*}
where $I_A(z)$ denotes the indicator function assuming $1$ if $z \in A$ and $0$ otherwise. 
\end{proof}

\begin{proof}[Proof of Theorem
\ref{index.proper}]

\vspace{0.5 cm}
\begin{enumerate}[(a)]
\item It follows directly from \eqref{WIKS2} and the fact that $W$ assumes values in $[0,1]$;
\item Since the random variable $W(d(P_1, P_2))$ is non negative, its expected value is $0$ if and only if it assumes $0$ almost surely;
\item The same argument of (b) applied to the non negative random variable $1 - W(d(P_1, P_2))$;
\item Consider $D_1$ and $D_2$ two random variables representing two posterior distributions for $d(P_1, P_2)$ such that $D_2$ is stochastically greater than $D_1$, i.e., $\mathbb{P}(D_1 \geq x) \leq \mathbb{P}(D_2 \geq x)$ for all $x > 0$. Since $\mathbb{E}[D_i] = \int_0^{\infty} \mathbb{P}(D_i \geq x) dx$, we have that $\mathbb{E}[D_1] \leq \mathbb{E}[D_2]$.
\end{enumerate}
\end{proof}

\begin{proof}[Proof of Theorem \ref{dt}]

For a decision $\delta(x, y) \in \mathbb{A}$, the posterior expected loss is given by
\[ \mathbb{E}^{x,y}\big[L\big((P_1, P_2), \delta(x,y) \big)\big]= \left\{ \begin{array}{ll}
c_0 \mathbb{E}^{x,y}[W(d(P_1, P_2))], & \mbox{ if $\delta(x,y) = 0$}, \\
c_1\big[1-\mathbb{E}^{x,y}[W(d(P_1, P_2))]\big], & \mbox{ if $\delta(x,y) = 1$}.\end{array} \right. \]
Thus, the Bayes rule is given by rejecting $H_0$ if and only if
\begin{align*}
WIKS(\mathcal{D}_{n,m}) = \mathbb{E}^{x, y}\big[W(d(P_1, P_2))\big] >  c_1/(c_1 + c_0).
\end{align*}
\end{proof}

\begin{thm}
\label{thm::boundDistanceExpected} \ 

\begin{enumerate}[(i)]
\item 
For any continuous distribution function $H$ (not necessarily being the generating mechanism associated
to $X$ or $Y$), 
$$Z^{n,m}(\mathbb{D}_{n,m})=\sup_{t \in [0,1]} \Bigg \lvert\dfrac{1}{K + n} \sum_{i=1}^n I_{[0, t]}(H(X_i)) - \dfrac{1}{K + m}\sum_{j=1}^m I_{[0, t]}(H(Y_j)) \Bigg \rvert.$$
\item 
$$\Bigg \lvert d\big(E_{\mathbb{D}_{n,m}}[P_{1}^*], E_{\mathbb{D}_{n,m}}[P_{2}^*]\big)- Z^{n,m}(\mathbb{D}_{n,m}) \Bigg \rvert  \leq \dfrac{K|m - n|}{(K + m)(K + n)}.$$
\item If both samples $X_1, \ldots, X_n$ and $Y_1, \ldots, Y_m$ have a common cumulative distribution function $H$, then the distribution of $Z^{n,m}(\mathbb{D}_{n,m})$  is invariant with respect to $H$.
\end{enumerate}
\end{thm}
\begin{proof}
\noindent (i)
Observe that
\begin{align*}
Z^{n,m}_{H}(\mathbb{D}_{n,m}) &:= \sup_{x \in \mathbb{R}} \Bigg \lvert\dfrac{n}{K + n} \dfrac{1}{n} \sum_{i=1}^n I_{(-\infty, x]}(X_i) - \dfrac{m}{K + m} \dfrac{1}{m}\sum_{j=1}^m I_{(-\infty, x]}(Y_j) \Bigg \rvert. \\
&= \sup_{x \in \mathbb{R}} \Bigg \lvert\dfrac{1}{K + n} \sum_{i=1}^n I_{[0, H(x)]}(H(X_i)) - \dfrac{1}{K + m}\sum_{j=1}^m I_{[0, H(x)]}(H(Y_j)) \Bigg \rvert. \\
&= \sup_{t \in [0,1]} \Bigg \lvert\dfrac{1}{K + n} \sum_{i=1}^n I_{[0, t]}(H(X_i)) - \dfrac{1}{K + m}\sum_{j=1}^m I_{[0, t]}(H(Y_j)) \Bigg \rvert,
\end{align*}
where in the last equality we used the continuity of $F$.

\noindent (ii)
Notice that $d\big(E[P_{1,n}^*], E[P_{2,m}^*]\big)$ can be expressed as
\begin{align*}
d\big(&E_{\mathbb{D}_{n,m}}[P_{1}^*], E_{\mathbb{D}_{n,m}}[P_{2}^*]\big) \\
&\stackrel{}{=} \sup_{x \in \mathbb{R}}  \Bigg \lvert \dfrac{K}{K + n} G(x) + \dfrac{1}{K + n} \sum_{i=1}^n I_{[0, H(x)]}(H(X_i)) -  \dfrac{K}{K + m} G(x) - \dfrac{1}{K + m}\sum_{j=1}^m I_{[0, H(x)]}(H(Y_j)) \Bigg \rvert \\
&= \sup_{x \in \mathbb{R}}  \Bigg \lvert \Bigg( \dfrac{K}{K + n} - \dfrac{K}{K + m} \Bigg) G(x) + \dfrac{1}{K + n} \sum_{i=1}^n I_{[0, H(x)]}(H(X_i)) - \dfrac{1}{K + m}\sum_{j=1}^m I_{[0, H(x)]}(H(Y_j)) \Bigg \rvert \\
&= \sup_{x \in \mathbb{R}}  \Bigg \lvert \Bigg( \dfrac{K}{K + n} - \dfrac{K}{K + m} \Bigg) (G \circ H^{-1})(H(x)) + \dfrac{1}{K + n} \sum_{i=1}^n I_{[0, H(x)]}(H(X_i)) - \dfrac{1}{K + m}\sum_{j=1}^m I_{[0, H(x)]}(H(Y_j)) \Bigg \rvert \\
&= \sup_{t \in [0,1]} \Bigg \lvert \Bigg( \dfrac{K}{K + n} - \dfrac{K}{K + m} \Bigg) (G \circ H^{-1})(t) + \dfrac{1}{K + n} \sum_{i=1}^n I_{[0, t]}(H(X_i)) - \dfrac{1}{K + m}\sum_{j=1}^m I_{[0, t]}(H(Y_j)) \Bigg \rvert.
\end{align*} 
Moreover,
\begin{align}
\label{eq:upper}
\sup_{t \in [0,1]}& \Bigg \lvert \Bigg( \dfrac{K}{K + n} - \dfrac{K}{K + m} \Bigg) (G \circ H^{-1})(t) + \dfrac{1}{K + n} \sum_{i=1}^n I_{[0, t]}(H(X_i)) - \dfrac{1}{K + m}\sum_{j=1}^m I_{[0, t]}(H(Y_j)) \Bigg \rvert  \notag \\
& \leq \sup_{t \in [0,1]} \Bigg \lvert \Bigg( \dfrac{K}{K + n} - \dfrac{K}{K + m} \Bigg) (G \circ H^{-1})(t) \Bigg \rvert
 + \sup_{t \in [0,1]} \Bigg \lvert\dfrac{1}{K + n} \sum_{i=1}^n I_{[0, t]}(H(X_i)) - \dfrac{1}{K + m}\sum_{j=1}^m I_{[0, t]}(H(Y_j)) \Bigg \rvert   \notag  \\
 &=\Bigg| \dfrac{K}{K + n} - \dfrac{K}{K + m} \Bigg|  \sup_{t \in [0,1]} \Bigg \lvert (G \circ H^{-1})(t) \Bigg \rvert
  + \sup_{t \in [0,1]} \Bigg \lvert\dfrac{1}{K + n} \sum_{i=1}^n I_{[0, t]}(H(X_i)) - \dfrac{1}{K + m}\sum_{j=1}^m I_{[0, t]}(H(Y_i)) \Bigg \rvert \notag \\
&= \dfrac{K|m - n|}{(K + m)(K + n)}
+ \sup_{t \in [0,1]} \Bigg \lvert\dfrac{1}{K + n} \sum_{i=1}^n I_{[0, t]}(H(X_i)) - \dfrac{1}{K + m}\sum_{j=1}^m I_{[0, t]}(H(Y_j)) \Bigg \rvert.
\end{align}
Also,
\begin{align}
\label{eq:lower}
\sup_{t \in [0,1]}& \Bigg \lvert \Bigg( \dfrac{K}{K + n} - \dfrac{K}{K + m} \Bigg) (G \circ H^{-1})(t) + \dfrac{1}{K + n} \sum_{i=1}^n I_{[0, t]}(H(X_i)) - \dfrac{1}{K + m}\sum_{j=1}^m I_{[0, t]}(H(Y_j)) \Bigg \rvert   \notag \\
&\geq\sup_{t \in [0,1]} \Bigg \lvert\dfrac{1}{K + n} \sum_{i=1}^n I_{[0, t]}(H(X_i)) - \dfrac{1}{K + m}\sum_{j=1}^m I_{[0, t]}(H(Y_j)) \Bigg \rvert
 - \Bigg \lvert \Bigg( \dfrac{K}{K + n} - \dfrac{K}{K + m} \Bigg) (G \circ H^{-1})(t) \Bigg \rvert    \notag  \\
&\geq \sup_{t \in [0,1]} \Bigg \lvert\dfrac{1}{K + n} \sum_{i=1}^n I_{[0, t]}(H(X_i)) - \dfrac{1}{K + m}\sum_{j=1}^m I_{[0, t]}(H(Y_j)) \Bigg \rvert 
 -   \Bigg| \dfrac{K}{K + n} - \dfrac{K}{K + m} \Bigg|  \sup_{t \in [0,1]} \Bigg \lvert (G \circ H^{-1})(t) \Bigg \rvert \notag \\
 &= \sup_{t \in [0,1]} \Bigg \lvert\dfrac{1}{K + n} \sum_{i=1}^n I_{[0, t]}(H(X_i)) - \dfrac{1}{K + m}\sum_{j=1}^m I_{[0, t]}(H(Y_j)) \Bigg \rvert 
 -   \dfrac{K|m - n|}{(K + m)(K + n)}.
 \end{align}
 
 The conclusion of (ii) follows  from Equations \ref{eq:upper} and \ref{eq:lower}.
 
 \noindent (iii) It  suffices to observe that  the hypothesis implies that $H(X_1), \ldots, H(X_n), H(Y_1), \ldots, H(Y_m) \stackrel{i.i.d}{\sim} U(0,1)$.
 
\end{proof}

\begin{Lemma}
\label{lemma:lowerandupper}
Under Assumption \ref{assump::uniform},
\begin{align*}
d(E_{\mathbb{D}_{n,m}}[P_{1}^*],E_{\mathbb{D}_{n,m}}[P_{2}^*]) \leq \mbox{WIKS}(\mathbb{D}_{n,m}) \leq  & d(E_{\mathbb{D}_{n,m}}[P_{1}^*],E_{\mathbb{D}_{n,m}}[P_{2}^*])+\\
+&E_{\mathbb{D}_{n,m}}[d(P_{1}^*,E_{\mathbb{D}_{n,m}}[P_{1}^*])] +E_{\mathbb{D}_{n,m}}[d(P_{2}^*,E_{\mathbb{D}_{n,m}}[P_{2}^*])]
\end{align*}
\end{Lemma}
\begin{proof}
To prove the first inequality, let $P$ and $Q$ be two cumulative distribution functions and let
$g_{Q}(P):=\sup_x |P(x)-Q(x)|$.
$g_{Q}$ is convex. Indeed,
\begin{align*}
g_{Q}(w_1P_{1}+w_2P_{2})&=\sup_x |w_1P_{1}(x)+w_2P_{2}(x)-w_1Q(x)-w_2Q(x)|   \\
&\leq 
\sup_x |w_1P_{1}(x)-w_1Q(x)|+\sup_x |w_2P_{2}(x)-w_2Q(x)|=w_1 g_{Q}(P_{1})+w_2 g_{Q}(P_{2}).
\end{align*}
Thus, by applying Jensen's inequality twice and using the independence of the processes,
\begin{align*}
E_{\mathbb{D}_{n,m}}[d(P_{1}^*,P_{2}^*)]&=E_{\mathbb{D}_{n,m}} [g_{P_{2}^*}(P_{1}^*)]=E_{\mathbb{D}_{n,m}}[E_{\mathbb{D}_{n,m},P_{2}^*}[g_{P_{2}^*}(P_{1}^*)]]\geq E_{\mathbb{D}_{n,m}}[g_{P_{2}^*}(E_{\mathbb{D}_{n,m},P_{2}^*}[P_{1}^*])] \\
 &=E_{\mathbb{D}_{n,m}}[g_{P_{2}^*}(E_{\mathbb{D}_{n,m}}[P_{1}^*])]= 
E_{\mathbb{D}_{n,m}}[g_{E_{\mathbb{D}_{n,m}}[P_{1}^*]}(P_{2}^*)]
\\
&\geq g_{E_{\mathbb{D}_{n,m}}[P_{1}^*]}(E_{\mathbb{D}_{n,m}}[P_{2}^*])=d(E_{\mathbb{D}_{n,m}}[P_{1}^*],E_{\mathbb{D}_{n,m}}[P_{2}^*])
\end{align*}

The second inequality follows from the triangle inequality.
\end{proof}

\begin{Lemma} 
\label{lemma:asconvergence}
Under Assumptions \ref{assump::bounded}
and \ref{assump::uniform},
$$ E_{\mathbb{D}_{n,m}}[d(P_{1}^*,E_{\mathbb{D}_{n,m}}[P_{1}^*])] \xrightarrow{n \longrightarrow \infty} 0 $$
and
$$ E_{\mathbb{D}_{n,m}}[d(P_{2}^*,E_{\mathbb{D}_{n,m}}[P_{2}^*])]    \xrightarrow{m \longrightarrow \infty} 0 $$
\end{Lemma}
\begin{proof}

Because
$$\sup_x  |P_{1,n}(x)-G_n(x)| \leq \left(\sqrt{3}\int_{\mathbb{R}} (P_{1,n}(x)-G_n(x))^2 dG_n(x)\right)^{2/3}$$
\citep[page 603]{donoho1988pathologies}
and by Jensen's inequality,
\begin{align}
\label{eq:upper1}
E\left[\sup_x  |P_{1,n}(x)-G_n(x)| \right] &\leq    E\left[\left(\sqrt{3}\int_{\mathbb{R}} (P_{1,n}(x)-G_n(x))^2 dG_n(x)\right)^{2/3} \right] \notag \\
&\leq
\left( E\left[\sqrt{3}\int_{\mathbb{R}} (P_{1,n}(x)-G_n(x))^2 dG_n(x)\right]\right)^{2/3}.
\end{align}
Let  
$\nu_2$ be the counting measure. 
The Radon-Nikodym derivative of
$G_n$ with respect to $(\nu_1+\nu_2)$ is
$$g_n(x):=\frac{dG_n}{d(\nu_1+\nu_2)}(x)=\frac{K}{K+n}g(x) \mathbb{I}_{\mathbb{R} \backslash A}(x) + \frac{1}{K+n}\sum_{i=1}^n \mathbb{I}_{x_i}(x),$$
where $A = \{x_1, \ldots, x_n\}$ 
\citep[Theorem 1]{gottardo2008markov}.

Assumption \ref{assump::bounded} implies that,
for every $\mathcal{D}^x_{n}$,
$g_n(x) \leq C$. 
It follows that  the inner part of the right hand-side of Equation \ref{eq:upper1} is
\begin{align*}
&E\left[\sqrt{3}\int_{\mathbb{R}} (P_{1,n}(x)-G_n(x))^2 dG_n(x)\right]
=E\left[\sqrt{3}\int_{\mathbb{R}} (P_{1,n}(x)-G_n(x))^2 g_n(x) d(\nu_1+\nu_2)(x)\right] \\
&\leq
C\sqrt{3} E\left[\int_{\mathbb{R}} (P_{1,n}(x)-G_n(x))^2 d(\nu_1+\nu_2)(x)\right]
= C\sqrt{3} \int_{\mathbb{R}} E\left[(P_{1,n}(x)-G_n(x))^2 \right]d(\nu_1+\nu_2)(x),
\end{align*}
where the last step follows from Tonelli's theorem. 
Thus,
\begin{align}
\label{eq:uppersup}
 E_{\mathbb{D}_{n}^X}\left[  d(P^*_{1}(x),E_{\mathbb{D}_{n}^X}[P^*_{1}]) \right]&= E\left[\sup_x  |P_{1,n}(x)-G_n(x)| \right]\notag \\
 &\leq \left( C \int_{\mathbb{R}} E\left[(P_{1,n}(x)-G_n(x))^2 \right]d(\nu_1+\nu_2)(x)\right)^{2/3} 
\end{align}

Moreover, for every $\mathbb{D}_n^x$,
$$E\left[(P_{1,n}(x)-G_n(x))^2 \right]=V\left[P_{1,n}(x)\right]=O(n^{-1}),$$
so that, for every $\mathbb{D}_n^x$, 
$$\lim_{n \longrightarrow \infty}  E\left[(P_{1,n}(x)-G_n(x))^2 \right] =0.$$
Because $E\left[(P_{1,n}(x)-G_n(x))^2 \right]\leq 2$, it follows from
the dominated convergence theorem that 
$$\lim_{n \longrightarrow \infty} \int_{\mathbb{R}} E\left[(P_{1,n}(x)-G_n(x))^2 \right]d(\nu_1+\nu_2)(x) =
\int_{\mathbb{R}} \lim_{n \longrightarrow \infty}  E\left[(P_{1,n}(x)-G_n(x))^2 \right]d(\nu_1+\nu_2)(x)=0.$$
Conclude from Equation \ref{eq:uppersup} that 
$$ \lim_{n \longrightarrow \infty} E_{\mathbb{D}_{n}^X}\left[ d(P^*_{1}(x),E_{\mathbb{D}_{n}^X}[P^*_{1}]) \right] =0.
$$

The second limit is analogous.
\end{proof}

\begin{thm}
\label{thm:limitjensen}
Under Assumptions \ref{assump::bounded}
and \ref{assump::uniform},
$$d(E_{\mathbb{D}_{n,m}}[P_{1}^*],E_{\mathbb{D}_{n,m}}[P_{2}^*]) - \mbox{WIKS}(\mathbb{D}_{n,m})
\xrightarrow{n,m \longrightarrow \infty} 0 
$$
\end{thm}
\begin{proof}
Lemma \ref{lemma:lowerandupper} implies that
$$0 \leq d(E_{\mathbb{D}_{n,m}}[P_{1}^*],E_{\mathbb{D}_{n,m}}[P_{2}^*]) - E_{\mathbb{D}_{n,m}}[d(P_{1}^*,P_{2}^*)] \leq E_{\mathbb{D}_{n,m}}[d(P_{1}^*,E_{\mathbb{D}_{n,m}}[P_{1}^*])] +E_{\mathbb{D}_{n,m}}[d(P_{2}^*,E_{\mathbb{D}_{n,m}}[P_{2}^*])].$$
The conclusion follows from taking the limit as $n,m \longrightarrow \infty$ and using Lemma \ref{lemma:asconvergence}.
\end{proof}

\begin{thm}
\label{thm::boundDistanceExpected2}
Under Assumptions \ref{assump::bounded}
and \ref{assump::uniform},
$$\Bigg \lvert \mbox{WIKS}(\mathbb{D}_{n,m})- Z^{n,m}(\mathbb{D}_{n,m}) \Bigg \rvert  \xrightarrow{n,m \longrightarrow \infty} 0, 
$$
where $Z^{n,m}(\mathbb{D}_{n,m})$ is defined in Theorem \ref{thm::boundDistanceExpected}.
\end{thm}
\begin{proof}
By Theorem \ref{thm::boundDistanceExpected}, 
\begin{align}
\label{eq:limitdist}
\Bigg \lvert d\big(E_{\mathbb{D}_{n,m}}[P_{1}^*], E_{\mathbb{D}_{n,m}}[P_{2}^*]\big)- Z^{n,m}(\mathbb{D}_{n,m}) \Bigg \rvert  \leq \dfrac{K|m - n|}{(K + m)(K + n)}
 \xrightarrow{n,m \longrightarrow \infty}  0
\end{align}
Now,
\begin{align*}
0 &\leq \Bigg \lvert \mbox{WIKS}(\mathbb{D}_{n,m})- Z^{n,m}(\mathbb{D}_{n,m}) \Bigg \rvert  \leq  \\
&\Bigg \lvert \mbox{WIKS}(\mathbb{D}_{n,m})-d\big(E_{\mathbb{D}_{n,m}}[P_{1}^*], E_{\mathbb{D}_{n,m}}[P_{2}^*]\big)+d\big(E_{\mathbb{D}_{n,m}}[P_{1}^*], E_{\mathbb{D}_{n,m}}[P_{2}^*]\big)- Z^{n,m}(\mathbb{D}_{n,m}) \Bigg \rvert  \\
&\leq
\Bigg \lvert \mbox{WIKS}(\mathbb{D}_{n,m})-d\big(E_{\mathbb{D}_{n,m}}[P_{1}^*], E_{\mathbb{D}_{n,m}}[P_{2}^*]\big) \Bigg \rvert  +
\Bigg \lvert d\big(E_{\mathbb{D}_{n,m}}[P_{1}^*], E_{\mathbb{D}_{n,m}}[P_{2}^*]\big)- Z^{n,m}(\mathbb{D}_{n,m})\Bigg \rvert 
 \xrightarrow{n,m \longrightarrow \infty}  0,  
\end{align*}
where the computation of the limit follows from 
Theorem \ref{thm:limitjensen} and Equation \ref{eq:limitdist}.
\end{proof}

\begin{proof}[Proof of Corollary \ref{cor:invariant}]

Fix $x \in \mathbb{R}$ and
assume that $H_X=H_Y=H$. By Theorem \ref{thm::boundDistanceExpected}
and the lower bound of Lemma \ref{lemma:lowerandupper},
\begin{align}
\label{eq:wiksupper}
P\left(\mbox{WIKS}(\mathbb{D}_{n,m}) \leq F^{-1}_{Z^{n,m}(\mathbb{U})}(x) \right) \leq  
P\left( Z^{n,m}(\mathbb{D}_{n,m})  \leq F^{-1}_{Z^{n,m}(\mathbb{U})}(x)+C_{n,m} \right),
\end{align}
where $C_{n,m}= \dfrac{K|m - n|}{(K + m)(K + n)}$. 
Let $K_{n,m}=\sqrt{mn/(m+n)}$.
Then 
\begin{align}
\label{eq:leftlimit}
K_{n,m}\times Z^{n,m}(\mathbb{D}_{n,m}) \xrightarrow[n,m \longrightarrow \infty]{D} Z,
\end{align}
where $Z$ is the Kolmogorov distribution \citep{raghavachari1973limiting}. 
Because under $H_0$ $Z^{n,m}(\mathbb{D}_{n,m}) \sim Z^{n,m}(\mathbb{U})$ and $F_Z$ is a continuous function, it also follows that
$$K_{n,m} \times F^{-1}_{Z^{n,m}(\mathbb{U})}(x) \xrightarrow[n,m \longrightarrow \infty]{D} F^{-1}_{Z}(x).$$
Moreover,  
$$K_{n,m} \times C_{n,m} \xrightarrow{n,m \longrightarrow \infty} 0,$$
and therefore
\begin{align}
\label{eq:rightlimit}
F^{-1}_{Z^{n,m}(\mathbb{U})}(x)+C_{n,m} \xrightarrow{n,m \longrightarrow \infty} F^{-1}_{Z}(x).
\end{align} 
Using Slutsky's Theorem, Equation \ref{eq:leftlimit} and Equation \ref{eq:rightlimit}, conclude that
\begin{align*}
&P\left( Z^{n,m}(\mathbb{D}_{n,m})  \leq F^{-1}_{Z^{n,m}(\mathbb{U})}(x)+C_{n,m} \right) \\
&=
P\left(K_{n,m}\times Z^{n,m}(\mathbb{D}_{n,m})  \leq K_{n,m}\times F^{-1}_{Z^{n,m}(\mathbb{U})}(x)+K_{n,m}\times C_{n,m} \right)
\xrightarrow[n,m \longrightarrow \infty]{D} F_Z\left(F^{-1}_{Z}(x)\right)=x.
\end{align*}
Conclude from Equation \ref{eq:wiksupper} that
$\lim_{n,m} P\left(\mbox{WIKS}(\mathbb{D}_{n,m}) \leq F^{-1}_{Z^{n,m}(\mathbb{U})}(x) \right) \leq x$.

Now, by Theorem \ref{thm::boundDistanceExpected}
and the upper bound of Lemma \ref{lemma:lowerandupper},
\begin{align*}
&P\left(\mbox{WIKS}(\mathbb{D}_{n,m}) \leq F^{-1}_{Z^{n,m}(\mathbb{U})}(x) \right) \geq  \\
&P\left( Z^{n,m}(\mathbb{D}_{n,m})  \leq F^{-1}_{Z^{n,m}(\mathbb{U})}(x)-C_{n,m}-E_{\mathbb{D}_{n,m}}[d(P_{1}^*,E_{\mathbb{D}_{n,m}}[P_{1}^*])] -E_{\mathbb{D}_{n,m}}[d(P_{2}^*,E_{\mathbb{D}_{n,m}}[P_{2}^*])] \right).
\end{align*}
Using the same strategy as in the upper bound,
conclude that
$$\lim_{n,m}  P\left(\mbox{WIKS}(\mathbb{D}_{n,m}) \leq F^{-1}_{Z^{n,m}(\mathbb{U})}(x) \right) \geq  x,$$
which concludes the proof of the theorem.
\end{proof}

\begin{Lemma}
\label{lemma:consistKS}
$$ Z^{n,m}(\mathbb{D}_{n,m}) \xrightarrow[n,m \longrightarrow \infty]{a.s.}  \sup_{x \in \mathbb{R}}|H_X(x)-H_Y(x)|$$
\end{Lemma}
\begin{proof}
It follows from the strong law of large numbers.
\end{proof}

\begin{proof}[Proof  of Corollary \ref{cor::consist}] 
It follows from Theorem \ref{thm::boundDistanceExpected2} and Lemma \ref{lemma:consistKS}.
\end{proof}

\begin{proof}[Proof of Corollary \ref{cor::consist2}] 

The first limit follows directly from Corollary \ref{cor:invariant}.
We now prove the second limit.
Corollary \ref{cor::consist} implies that if $H_1$ holds,
$$\mbox{WIKS}(\mathbb{D}_{n,m})  \xrightarrow[n,m \longrightarrow \infty]{P}  a,$$
and hence 
\begin{align}
\label{eq::wiksH0}
\mbox{WIKS}(\mathbb{D}_{n,m})-a  \xrightarrow[n,m \longrightarrow \infty]{P}  0
\end{align}
for some $a>0$. Moreover, by Lemma \ref{lemma:consistKS},
$$ Z^{n,m}(\mathbb{U}) \xrightarrow[n,m \longrightarrow \infty]{P}  0,$$
which implies that 
\begin{align}
\label{eq::inverselimit}
F^{-1}_{Z^{n,m}(\mathbb{U})}(1-\alpha)  \xrightarrow[n,m \longrightarrow \infty]{P}   0.
\end{align} 
It follows from Equations \ref{eq::wiksH0} and \ref{eq::inverselimit} that
for every $\epsilon>0,$
\begin{align*}
P\left(\mbox{WIKS}(\mathbb{D}_{n,m})-F^{-1}_{Z^{n,m}(\mathbb{U})}\geq a - \epsilon\right) \geq P\left(|\mbox{WIKS}(\mathbb{D}_{n,m})-a-F^{-1}_{Z^{n,m}(\mathbb{U})}(1-\alpha)|\leq \epsilon\right)\longrightarrow 1.
\end{align*}
The conclusion follows from taking $\epsilon=a$.
\end{proof}
  
\end{document}